\documentclass[12pt]{amsart}
\usepackage{amssymb,latexsym,xypic}

\theoremstyle{plain} 
\newtheorem{theorem}[equation]{Theorem}
\newtheorem{corollary}[equation]{Corollary}
\newtheorem{lemma}[equation]{Lemma}
\newtheorem{proposition}[equation]{Proposition}

\theoremstyle{definition}
\newtheorem{definition}[equation]{Definition}

\theoremstyle{remark}
\newtheorem{remark}[equation]{Remark}
\newtheorem{example}[equation]{Example}

\newcommand{\BV}{B}
\newcommand{\sectiontriviale}{s}
\def\build#1_#2^#3{\mathrel{\mathop{\kern0pt#1}\limits_{#2}^{#3}}}

\begin{document}
\title{\bf A Batalin-Vilkovisky algebra morphism from double loop spaces to free loops}
\author{Luc Menichi}
\address{UMR 6093 associ\'ee au CNRS\\
Universit\'e d'Angers, Facult\'e des Sciences\\
2 Boulevard Lavoisier\\49045 Angers, FRANCE}
\email{firstname.lastname at univ-angers.fr}
\begin{abstract}
Let $M$ be a compact oriented $d$-dimensional smooth manifold
and $X$ a topological space.
Chas and Sullivan~\cite{Chas-Sullivan:stringtop} have defined a structure of Batalin-Vilkovisky
algebra on $\mathbb{H}_*(LM):=H_{*+d}(LM)$.
Getzler~\cite{Getzler:BVAlg} has defined a structure of Batalin-Vilkovisky
algebra on the homology of the pointed double loop space of $X$,
$H_*(\Omega^2 X)$.
Let $G$ be a topological monoid with a homotopy inverse.
Suppose that $G$ acts on $M$.
We define a structure of Batalin-Vilkovisky algebra
on $H_*(\Omega^2BG)\otimes\mathbb{H}_*(M)$ extending the Batalin-Vilkovisky algebra
of Getzler on $H_*(\Omega^2BG)$.
We prove that the morphism of graded algebras
$$H_*(\Omega^2BG)\otimes\mathbb{H}_*(M)\rightarrow\mathbb{H}_*(LM)$$
defined by Felix and Thomas~\cite{Felix-Thomas:monsefls}, is in fact a morphism of
Batalin-Vilkovisky algebras. In particular, if $G=M$ is a connected compact Lie group,
we compute the Batalin-Vilkovisky algebra
$\mathbb{H}_*(LG;\mathbb{Q})$.
\end{abstract}
\maketitle

\section{Introduction}
We work over an arbitrary principal ideal domain ${\Bbbk}$.

Algebraic topology gives us two sources of Batalin-Vilkovisky algebras (Definition~\ref{def BV algebre}): Chas Sullivan string topology~\cite{Chas-Sullivan:stringtop}
and iterated loop spaces.
More precisely, let $X$ be be a pointed topological space, extending the work of Cohen~\cite{Cohen-Lada-May:homiterloopspaces}, Getzler\cite{Getzler:BVAlg} has shown that the homology $H_*(\Omega^2 X)$ of the double
pointed loopspace on $X$ is a Batalin-Vilkovisky algebra.
Let $M$ be a closed oriented manifold of dimension $d$. Denote by $LM:=\text{map}(S^1,M)$ the free loop space on $M$.
By Chas and Sullivan~\cite{Chas-Sullivan:stringtop}, the (degree-shifted) homology
$\mathbb{H}_*(LM):=H_{*+d}(LM)$ of $LM$ has also a Batalin-Vilkovisky algebra.
In this paper, we related this two a priori very different Batalin-Vilkovisky
algebras.

Initially, our work started with the following Theorem of Felix and Thomas.
\vspace{.3cm}
\begin{theorem}~\cite[Theorem 1]{Felix-Thomas:monsefls}\label{Felix et Thomas}
Let $G$ be a topological monoid acting on a smooth compact oriented
manifold $M$.
Consider the map $\Theta_{G,M}:\Omega G\times M\rightarrow LM$
which sends the couple $(w,m)$ to the free loop $t\mapsto w(t).m$.
The map induced in homology
$$\mathbb{H}_*(\Theta_{G,M}):H_*(\Omega G)\otimes\mathbb{H}_*(M)
\rightarrow \mathbb{H}_*(LM)$$
is a morphism of commutative graded algebras.   
\end{theorem}
In~\cite{Felix-Thomas:monsefls}, F\'elix and Thomas stated this theorem with
$G=aut_1(M)$, the monoid of self equivalences homotopic to the
identity.
But their theorem extends to any topological monoid $G$
since any action on $M$ factorizes through $\mbox{Hom}(M,M)$
and $aut_1(M)$ is the path component of the identity in
$\mbox{Hom}(M,M)$.
Note also that the assumptions of~\cite{Felix-Thomas:monsefls}
``${\Bbbk}$ is a field '' and ``$M$ is simply connected''
are not necessary.

If $G=M$ is a Lie group, $\Theta_{G,G}:\Omega G\times G\rightarrow LG$
is a homeomorphism. Therefore we recover the following well-known isomorphism
(See for example~\cite[Proof of Theorem 10]{menichi:stringtopspheres}).
\begin{corollary}\label{loop algebra des groupes de Lie}
Let $G$ be a path-connected compact Lie group.
The algebra $\mathbb{H}_*(LG)$ is isomorphic to the tensor
product of algebras $H_*(\Omega G)\otimes\mathbb{H}_*(G)$.
\end{corollary}
The following is our main theorem.
\begin{theorem}(Theorem~\ref{Structure BV sur le produit tensoriel})
Assume the hypothesis of Theorem~\ref{Felix et Thomas}.
Suppose moreover that $G$ has a homotopy inverse and that
$H_*(\Omega G)$ is torsion free.
Then the tensor product of algebras $H_*(\Omega G)\otimes \mathbb{H}_*(M)$
is a Batalin-Vilkovisky algebra and the morphism
$$\mathbb{H}_*(\Theta_{G,M}):
H_*(\Omega G)\otimes\mathbb{H}_*(M)\rightarrow \mathbb{H}_*(LM)$$
defined by Felix and Thomas,
is a morphism of Batalin-Vilkovisky algebras.
\end{theorem}
\begin{remark}
We assume that the homology with coefficients in ${\Bbbk}$, $H_*(\Omega G)$, is torsion free
(this hypothesis is of course satisfied if ${\Bbbk}$ is a field), since we want $H_*(\Omega G)$
 to have a diagonal.
\end{remark}
This Batalin-Vilkovisky algebra $H_*(\Omega G)\otimes\mathbb{H}_*(M)$
contains $H_*(\Omega G)$ as sub Batalin-Vilkovisky algebra (Corollary~\ref{comparaison delta}).
Denote by $BG$ the classifying space of $G$.
We show that the sub Batalin-Vilkovisky algebra $H_*(\Omega G)$
is isomorphic to the Batalin-Vilkovisky algebra $H_*(\Omega^2 BG)$
introduced by Getzler (Proposition~\ref{lacets pointe d'un groupe est BV}).
So finally ,we have obtained a morphism of Batalin-Vilkovisky
(Theorem~\ref{morphism de BV algebres des lacets doubles vers lacets libres})
$$H_*(\Omega^2 BG)\rightarrow \mathbb{H}_*(LM)$$
from the homology of the double pointed loop space on $BG$
to the free loop space homology on $M$.

Assuming that $H_*(\Omega aut_1(M))$ is torsion free,
Theorem~\ref{Structure BV sur le produit tensoriel} can also be applied to the monoid of
self equivalences $aut_1(M)$ as we shall now explain.
Since a smooth manifold has a CW-structure, $M$ is a finite
CW-complex. So by a theorem of Milnor~\cite{milnor:spacesCW}, $\mbox{Hom}(M,M)$
has the homotopy type of a CW-complex $K$.
Therefore $aut_1 M$ has the homotopy type of a path component
of $K$.
Recall that a path-connected homotopy associative $H$-space which has the homotopy of a CW-complex
has naturally a homotopy inverse~\cite[Chapter X. Theorem 2.2]{Whitehead:eltsoht},
i.e.~is naturally a {\it $H$-group}~\cite[p. 35]{Spanier:livre}.
Therefore $aut_1 M$ has a $H$-group structure.
\begin{remark}
In~\cite{Felix-Thomas:monsefls}, Felix and Thomas posed the following problem:
Is $$\mathbb{H}_*(\Theta_{aut_1(M),M}):H_*(\Omega
aut_1(M))\otimes\mathbb{H}_*(M)
\rightarrow\mathbb{H}_*(LM)$$
always surjective?
The answer is no.
Take $M:=S^{2n}$, $n\geq 1$. Rationally $aut_1 S^{2n}$ has the same
homotopy type as $S^{4n-1}$. Therefore
$H_*(\Omega aut_1 S^{2n};\mathbb{Q})\otimes H_*(S^{2n};\mathbb{Q})$
is concentrated in even degree. It is easy to see that
$H_*(LS^{2n};\mathbb{Q})$ is not trivial in odd degree.
So $H_*(\Theta_{aut_1(S^{2n}),S^{2n}})$ cannot be surjective.
 \end{remark}

We give now the plan of the paper:

{\bf Section 2:}
We give the definition and some examples of Batalin-Vilkovisky algebras.

{\bf Section 3:}
We compare the $S^1$-action on $\Omega^2 BG$ to a $S^1$-action defined on $\Omega G$.
Therefore the obvious isomorphism of algebras $H_*(\Omega G)\rightarrow
H_*(\Omega^2 BG)$ is in fact an isomorphim of Batalin-Vilkovisky algebras.

{\bf Section 4:} The Batalin-Vilkovisky algebra
$H_*(\Omega G)\otimes\mathbb{H}_*(M)$ is introduced and related to $\mathbb{H}_*(LM)$
(Theorem~\ref{Structure BV sur le produit tensoriel}). Then extending the Theorem
of Hepworth for Lie groups, we show how to compute its operator $\BV$
(Proposition~\ref{operateur BV sur le produit lacets espace}).
This Batalin-Vilkovisky algebra structure on $H_*(\Omega G)\otimes\mathbb{H}_*(M)$
is completely determined by its two sub Batalin-Vilkovisky algebras 
$H_*(\Omega G)$ and $\mathbb{H}_*(M)$
(Corollary~\ref{comparaison delta} and
Corollary~\ref{homologie de variete sous BV algebre})
and by the bracket between them (Corollary~\ref{expression simple du crochet}).

{\bf Section 5:} Extending the computations of the Batalin-Vilkovisky algebras
$\mathbb{H}_*(LS^1)$ and $\mathbb{H}_*(LS^3)$, we compute the Batalin-Vilkovisky
algebras $H_*(\Omega S^1)\otimes\mathbb{H}_*(M)$ and
$H_*(\Omega S^3)\otimes\mathbb{H}_*(M)$.

{\bf Section 6:} Using a second result of Felix and Thomas in~\cite{Felix-Thomas:monsefls}, we show that $\pi_{\geq 1}(\Omega aut_1(M))\otimes \mathbb{Q}$
is always a sub Lie algebra of $\mathbb{H}_*(LM;\mathbb{Q})$.  

{\bf Section 7:} This section is devoted to calculations of the Batalin-Vilkovisky
algebra $H_*(\Omega G)\otimes\mathbb{H}_*(M)$ over the rationals.
In particular, for any path-connected compact Lie group $G$, we compute
$\mathbb{H}_*(LG:\mathbb{Q})$.

{\bf Section 8:} This section is an appendix on split fibrations.

{\em Acknowledgment:}
We would like to thank Yves F\'elix and Jean-Claude Thomas for helpful comments and
for explaining us~\cite{Felix-Thomas:monsefls}.
I would like also to thank Jim Stasheff, for sending me lists of corrections
for several papers (including this one).
I would like especially to thank the referee for his serious job that led me to greatly
inprove this paper.
\section{Batalin-Vilkovisky algebras}
\begin{definition}\label{def BV algebre}
A Batalin-Vilkovisky algebra is a commutative graded algebra
$A$ equipped with an action of the exterior algebra $E(i_1)=H_*(S^{1})$
$$H_*(S^{1})\otimes A\rightarrow A$$
$$i_1\otimes a\mapsto i_1\cdot a\quad\text{
  denoted}\quad\BV(a)$$
such that
\begin{multline*}
\BV(abc)=\BV(ab)c+(-1)^{\vert a\vert} a\BV(bc)+(-1)^{(\vert a\vert -1)\vert b\vert}b\BV(ac)\\
-(\BV a)bc-(-1)^{\vert a\vert}a(\BV b)c
-(-1)^{\vert a\vert +\vert b\vert} ab(\BV c).
\end{multline*}

The bracket $\{\;,\;\}$ of degree $+1$ defined
by 
$$\{a,b\}=(-1)^{\vert a\vert}\left(\BV(ab)-(\BV a)b-(-1)^{\vert
  a\vert}a(\BV b)\right)$$
for any $a$, $b\in A$ satisfies the Poisson relation:
For any $a$, $b$ and $c\in A$,
\begin{equation}\label{Poisson relation}
\{a,bc\}=\{a,b\}c+(-1)^{(\vert a\vert-1)\vert b\vert}b\{a,c\}.
\end{equation}
\end{definition}
Koszul~\cite[p.~3]{Koszul:Crochet} (See also~\cite[Proposition
  1.2]{Getzler:BVAlg} or~\cite{Penkava-Schwarz:algebraic})
has shown that $\{\;,\;\}$ is a Lie
  bracket
and therefore that a Batalin-Vilkovisky algebra is a Gerstenhaber algebra.
\begin{example}(Tensor product of Batalin-Vilkovisky algebras)\label{produit tensoriel BV}
Let $A$ and $A'$ be two Batalin-Vilkovisky algebras.
Denote by $B_A$ and $B_{A'}$ their respective operators.
Consider the tensor product of algebras $A\otimes A'$.
Consider the operator $B_{A\otimes A'}$ on $A\otimes A'$ given by
$$B_{A\otimes A'}(x\otimes y)=B_A(x)\otimes y+(-1)^{\vert x\vert}x\otimes B_{A'}(y)$$
for $x\in A$ and $y\in A'$.
Then $A\otimes A'$ equipped with $B_{A\otimes A'}$ is a Batalin-Vilkovisky algebra.

Let $X$ and $X'$ be two pointed topological spaces.
Let $M$ and $M'$ be two compact oriented smooth manifolds.
It is easy to check that the Kunneth morphisms
$$
H_*(\Omega^2 X)\otimes H_*(\Omega^2 X')\rightarrow H_*(\Omega^2(X\times X'))
$$
$$
\mathbb{H}_*(LM)\otimes\mathbb{H}_*(LM')\rightarrow \mathbb{H}_*(L(M\times M'))
$$
are morphisms of Batalin-Vilkovisky algebras.
\end{example}
\section{The circle action on the pointed loops of a group}\label{action du cercle sur les lacets d'un groupe}
In this section, we define an action up to homotopy of the circle $S^1$
on the pointed loops $\Omega G$ of a $H$-group $G$.
When $G$ is a monoid, we show (Propositions~\ref{actions sur lacets doubles homotopes} and~\ref{lacets pointe d'un groupe est BV}) that the algebra $H_*(\Omega G)$
equipped with the operator induced by this action is a Batalin-Vilkovisky
algebra isomorphic to the Batalin-Vilkovisky algebra $H_*(\Omega^2 BG)$
introduced by Getzler~\cite{Getzler:BVAlg}.
Therefore, in this paper, instead of working with $H_*(\Omega^2 BG)$, we always
consider $H_*(\Omega G)$.

Consider the free loop fibration $\Omega X\buildrel{j}\over\hookrightarrow LX
\buildrel{ev}\over\twoheadrightarrow X$. The evaluation map
$ev:LX\twoheadrightarrow X$ admits a trivial
section $\sectiontriviale:X\hookrightarrow LX$.

Suppose now that $X$ is a $H$-group $G$. The following is easy to check if $X$ is a group.
For the $H$-group case, see appendix~\ref{semi-direct}.
Consider the map $r:LG\rightarrow \Omega G$ unique up to homotopy
such that the composite
$$j\circ r:LG\buildrel{r}\over\twoheadrightarrow
\Omega G\buildrel{j}\over\hookrightarrow LG$$
is homotopic to the map that sends the loop $l\in LG$
to the loop $t\in I\mapsto l(t)l(0)^{-1}$.
The map $\Theta_{G,G}:\Omega G\times G\rightarrow LG$ that maps $(w,g)$ to the free loop
$t\mapsto w(t)g$ is a homotopy equivalence.
Its homotopy inverse is the map $LG\rightarrow \Omega G\times G$,
$l\mapsto (r(l),l(0))$.
In particular, $r:LG\twoheadrightarrow \Omega G$ is a retract of $j$ and the
composite $r\circ\sectiontriviale$ is homotopically trivial.
\begin{theorem}(Compare with~\cite[Proposition 28]{menichi:stringtopspheres})\label{comparaison action du cercle sur lacets libres et lacets doubles}
Let $X$ be a topological space. The retract
$r:L\Omega X\twoheadrightarrow\Omega^2 X$ is a morphism of $S^1$-spaces
up to homotopy (i.e.~in the homotopy category of spaces).
\end{theorem}
Theorem~\ref{comparaison action du cercle sur lacets libres et lacets
 doubles} follows from Propositions~\ref{definition de l'action}i),
\ref{section est equivariante}
and \ref{actions sur lacets doubles homotopes}.
In~\cite[Proposition 28]{menichi:stringtopspheres}), the same theorem is proved but the
retract was defined in a different way.
\vspace{.3cm}
\begin{proposition}~\label{definition de l'action}

i) Let $G$ be a $H$-group.
Then $\Omega G$ is equipped with an action of $S^{1}$
up to homotopy.

ii) Let $G$ be a $H$-group acting up to homotopy on a topological
space $X$. Then $\Omega G\times X$ is equipped with an action of $S^{1}$
up to homotopy such that
 $\Theta_{G,X}:\Omega G\times X\rightarrow LX$
is a morphism of $S^1$-spaces up to homotopy.
\end{proposition}
\begin{proof}

i) We define the action of $S^{1}$ on $\Omega G$ as the composition of
$$S^{1}\times\Omega G\buildrel{S^{1}\times j}\over\hookrightarrow
S^{1}\times LG\buildrel{action}\over\rightarrow LG
\buildrel{r}\over\twoheadrightarrow\Omega G.$$

ii) We define the action of $S^1$ on $\Omega G\times X$ by
$s.(w,x):=(s.w,w(s).x)$, $s\in S^1$, $w\in\Omega$ and $x\in X$.
Here $s.w$ denote the action of $s\in S^1$ on $w\in\Omega G$ given in i).
It is easy to see that $\Theta_{G,X}$ is a morphism of $S^1$-spaces
up to homotopy. 
\end{proof}
\begin{proposition}\label{fonctorialite de l'action}
Let $G_1$ and $G_2$ be two $H$-groups.
Let $f:(G_1,e)\rightarrow (G_2,e)$ be an homomorphism of $H$-spaces
(in the sense of~\cite[p. 35]{Spanier:livre}).
Then $\Omega f:\Omega(H_1,e)\rightarrow \Omega(H_2,e)$
is a morphism of $S^{1}$-spaces up to homotopy.
\end{proposition}
\begin{proof}
A homomorphism of $H$-spaces between $H$-groups is necessarily
a homomorphism of $H$-groups.
The $S^1$-structure on $\Omega G$ is clearly functorial with respect
to homomorphisms of $H$-groups.
Therefore the $S^1$-structure on $\Omega G$ depends functorially only
on the multiplication of $G$.
\end{proof}
\begin{proposition}\label{section est equivariante}
The retract
$r:LG\twoheadrightarrow\Omega G$
is a morphism of $S^{1}$-spaces up to homotopy.
\end{proposition}
\begin{proof}
Denote by $p_{\Omega G}:\Omega G\times G\rightarrow\Omega G$ the projection
on the first factor.
By definition of $r$, the diagram

$$\xymatrix{
LG \ar[dr]_{r}& & \Omega G\times G
\ar[ll]_{\Theta_{G,G}}^{\simeq}
\ar[dl]^{p_{\Omega G}}\\
&\Omega G
}$$
commutes up to homotopy.
By Proposition~\ref{definition de l'action} ii),
$\Theta_{G,G}$ is a morphism of $S^1$-spaces up to homotopy.
By definition of the action of the circle $S^1$ on $\Omega G\times G$,
the projection $p_{\Omega G}$ is also a morphism of
$S^1$-spaces up to homotopy.
\end{proof}
\begin{proposition}\label{actions sur lacets doubles homotopes}
Let $(X,*)$ be a pointed space. The action of $S^{1}$ up to homotopy
on $\Omega G$ when $G=\Omega X$ given by Proposition~\ref{definition
  de l'action} is homotopic to the action of $S^{1}$ on
$map\left((E^{2},S^{1}),(X,*)\right)$ given by rotating the disk $E^{2}$.
\end{proposition}
\begin{proof}
The map $I\times S^1\rightarrow E^2$ sending $(t,x)\in\mathbb{R}\times\mathbb{C}$ to the barycenter of 
$x$ with weight $t$ and of $1$ with weight $1-t$,
gives the canonical homeomorphism
$$\theta:map\left((E^{2},S^{1}),(X,*)\right)\buildrel{\cong}\over\rightarrow
\Omega\Omega X$$
defined by:
 for $f\in map\left((E^{2},S^{1}),(X,*)\right)$, $\theta(f)$ is the
map sending $x\in S^1\subset\mathbb{C}$ to the loop
$t\in I\mapsto f(tx+1-t)$.

Since $j$ is a monomorphism in the homotopy category, to see that
$\theta$ commutes with the $S^1$-action up to homotopy, it suffices to see
that the two maps
$j\circ action\circ (S^1\times\theta)$ and $j\circ\theta\circ action$
are homotopic.

The adjoint of $j\circ\theta\circ action$ is the map
$$S^1\times map\left((E^{2},S^{1}),(X,*)\right)\times S^1\times I\rightarrow X$$
$$(e^{is},f,x,t)\mapsto f(e^{is}tx+(1-t)e^{is}).
$$
By definition of $r$, the adjoint of $j\circ action\circ (S^1\times\theta)$ is the map
$$S^1\times map\left((E^{2},S^{1}),(X,*)\right)\times S^1\times I\rightarrow X$$
$$(e^{is},f,x,t)\mapsto\begin{cases}
f\left((1-2t)e^{is}+2t\right) &\text{if $t\leq\frac{1}{2}$},\\
f\left((2t-1)e^{is}x+2-2t\right)& \text{if $t\geq\frac{1}{2}$}.
\end{cases}
$$
The maps of pairs of spaces
$$(S^1\times S^1\times I, S^1\times S^1\times \{0,1\})
\rightarrow (E^2,S^1)$$
$$(e^{is},x,t)\rightarrow e^{is}tx+(1-t)e^{is}$$
and
$$
(e^{is},x,t)\mapsto\begin{cases}
(1-2t)e^{is}+2t &\text{if $t\leq\frac{1}{2}$},\\
(2t-1)e^{is}x+2-2t& \text{if $t\geq\frac{1}{2}$}.
\end{cases}
$$
are homotopic: to construct the homotopy, fill the triangle of vertices
$e^{is}$, $e^{is}x$ and $1$.
Therefore $j\circ action\circ (S^1\times\theta)$ and $j\circ\theta\circ action$
are homotopic.
\end{proof}
\begin{proposition}\label{lacets pointe d'un groupe est BV}
Let $G$ be a topological monoid which is also a $H$-group.
Then the algebra
$H_*(\Omega G)$ equipped with the $H_*(S^1)$-module structure given
by Proposition~\ref{definition de l'action}i) is a Batalin-Vilkovisky algebra isomorphic to the
 Batalin-Vilkovisky algebra $H_*(\Omega^2 BG)$ given by~\cite{Getzler:BVAlg}.
\end{proposition}
\begin{proof}
Consider the classifying space of $G$, $BG$.
There exists a homomorphism of $H$-spaces
$h:G\buildrel{\simeq}\over\rightarrow\Omega BG$
which is a weak homotopy equivalence since $\pi_0(G)$ is a group.
By Propositions~\ref{fonctorialite de l'action}
and~\ref{actions sur lacets doubles homotopes},
$\Omega h:\Omega G\buildrel{\simeq}\over\rightarrow\Omega^2 BG$
is a morphism of $S^1$-spaces up to homotopy.
Therefore, 
$H_*(\Omega h)$ is both an isomorphism of graded algebras and
of $H_*(S^1)$-modules. Since $H_*(\Omega^2 BG)$ is a
Batalin-Vilkovisky algebra~\cite{Getzler:BVAlg}, 
$H_*(\Omega G)$ is also a Batalin-Vilkovisky algebra.
\end{proof}
\section{The Batalin-Vilkovisky algebra $H_*(\Omega^2BG)\otimes\mathbb{H}_*(M)$}
This section is the heart of the paper.
We show (Theorem~\ref{Structure BV sur le produit tensoriel}) that the tensor
product of algebras $H_*(\Omega G)\otimes\mathbb{H}_*(M)$ equipped with
an operator $B_{\Omega G\times M}$ is a Batalin-Vilkovisky algebra related to
the Batalin-Vilkovisky algebra $\mathbb{H}_*(LM)$ of Chas and Sullivan.
Extending a result of Hepworth (Corollary~\ref{string homology Lie group}),
we give an explicit formula for this operator $B_{\Omega G\times M}$.
We deduce then the morphism of Batalin-Vilkovisky algebra from $H_*(\Omega G)$
to $\mathbb{H}_*(LM)$ (Corollary~\ref{comparaison delta}).

Let us start by giving a short proof of the Felix and Thomas theorem.
\begin{proof}[Proof of Theorem~\ref{Felix et Thomas}]
Since $e.m=m$ for $m\in M$,
the map 
$\Theta_{G,M}:\Omega G\times M\rightarrow LM$
is a morphism of fiberwise monoids from the projection map
$p_M:\Omega G\times M\twoheadrightarrow M$ to the evaluation map
$ev:LM\twoheadrightarrow M$.
Therefore by~\cite[part 2) of Theorem 8.2]{gruher-salvatore:genstringtopop}, the composite
$$
\mathbb{H}_*(\Theta_{G,M}):H_*(\Omega G)\otimes\mathbb{H}_*(M)
\rightarrow H_{*+d}(\Omega G\times M)
\rightarrow\mathbb{H}_*(LM)  
$$
is a morphism of graded algebras.
\end{proof}
\begin{theorem}\label{Structure BV sur le produit tensoriel}
Let $G$ be a topological monoid with a homotopy inverse, acting
on a smooth compact oriented manifold $M$.
Assume that $H_*(\Omega G)$ is torsion free.
Consider the operator $B_{\Omega G\times M}$ given by the action of $H_*(S^1)$ on $H_*(\Omega G)\otimes\mathbb{H}_*(M)$
given by Proposition~\ref{definition de l'action} ii).
Then the tensor product of algebras $H_*(\Omega G)\otimes \mathbb{H}_*(M)$
equipped with $B_{\Omega G\times M}$ is a Batalin-Vilkovisky algebra such that
$$
H_*(\Theta_{G,M}):H_*(\Omega G)\otimes\mathbb{H}_*(M)\rightarrow \mathbb{H}_*(LM)
$$
is a morphim of Batalin-Vilkovisky algebras.
\end{theorem}
\begin{remark}\label{theta homeo dans le cas des groupes de Lie}
Suppose that $G=M$ is a Lie group. Then, with Corollary~\ref{loop algebra des groupes de Lie},
since $\Theta_{G,G}$ is an $S^1$-equivariant
homeomorphism, Theorem~\ref{Structure BV sur le produit tensoriel} is obvious.
\end{remark}
\begin{proof}
Since $H_*(\Omega G)$ is assumed to be torsion free, the Kunneth morphism
is an isomorphism.
Therefore the algebra $H_*(\Omega G)\otimes \mathbb{H}_*(LM)$
can be identified with the algebra $H_{*+d}(\Omega G\times LM)$.
For $w\in\Omega G$ and $l\in LM$, denote by $w*l$ the free loop on $M$ defined by
$
w*l(t)=w(t).l(t)
$, for $t\in S^1$.
Let $\Phi_{G,M}:\Omega G\times LM\rightarrow \Omega G\times LM$
be the map sending $(w,l)$ to $(w,w*l)$.
Since $G$ has an homotopy inverse, $\Phi_{G,M}$ is a homotopy equivalence.
Since $\Phi_{G,M}$ is a morphism of fiberwise monoids,
by~\cite[part 2) of Theorem 8.2]{gruher-salvatore:genstringtopop},
$$
H_*(\Phi_{G,M}):H_*(\Omega G)\otimes \mathbb{H}_*(LM)\rightarrow  H_*(\Omega G)\otimes \mathbb{H}_*(LM)
$$
is an isomorphism of algebras.

Consider the action of $S^1$ on $\Omega G$ up to homotopy given by
Proposition~\ref{definition de l'action} i) and the action of $S^1$ on $LM$
given by rotation of the loops. Consider the induced {\it diagonal} action of $S^1$
on $\Omega G\times LM$. Explicitly, if $G$ is a group, the diagonal action of $s\in S^1$ on
$(w,l)\in \Omega G\times LM$ is simply given by the pointed loop
$t\mapsto w(t+s)w(s)^{-1}$ and the free loop $t\mapsto l(t+s)$.

Consider the {\it twisted} action of $S^1$ on $\Omega G\times LM$ defined
by $s.(w,l)=(s.w,t\mapsto w(s).l(t+s))$ where $s.w$ is the action of $s\in S^1$
on $w\in\Omega G$ given by Proposition~\ref{definition de l'action} i).
With respect to the twisted action on the source and the diagonal action on the target,
$\Phi_{G,M}$ is a morphism of $S^1$-spaces up to homotopy.

The algebra $H_*(\Omega G)\otimes\mathbb{H}_*(LM)$ equipped with the $H_*(S^1)$-module structure
given by the diagonal action is the tensor product of the Batalin-Vilkovisky algebra
$H_*(\Omega G)$ given by Proposition~\ref{lacets pointe d'un groupe est BV}
and of the Batalin-Vilkovisky algebra $\mathbb{H}_*(LM)$ given by
Chas and Sullivan~\cite{Chas-Sullivan:stringtop}. Therefore by Example~\ref{produit tensoriel BV},
it
is a Batalin-Vilkovisky algebra.

Since the isomorphism $H_*(\Phi_{G,m})$ is both a morphism of algebras and a morphism
of $H_*(S^1)$-modules, $H_*(\Omega G)\otimes\mathbb{H}_*(LM)$ equipped with the 
$H_*(S^1)$-module structure given by the twisted action is also a Batalin-Vilkovisky algebra.

Consider the trivial section $\sectiontriviale:M\hookrightarrow LM$ mapping $x\in M$
to the free loop constant on $x$.
It is well-known that $\mathbb{H}_*(\sectiontriviale):\mathbb{H}_*(M)\rightarrow\mathbb{H}_*(LM)$
is a morphism of algebras.
The map $\Omega G\times \sectiontriviale:\Omega G\times M\rightarrow \Omega G\times LM$
is $S^1$-equivariant with respect to the $S^1$-action on $\Omega G\times M$
given by Proposition~\ref{definition de l'action} ii) and to the twisted $S^1$-action
on  $\Omega G\times LM$.
Therefore, since $H_*(\Omega G)\otimes\mathbb{H}_*(\sectiontriviale)$ is injective, $H_*(\Omega G)\otimes \mathbb{H}_*(M)$
is a Batalin-Vilkovisky algebra and
$H_*(\Omega G)\otimes\mathbb{H}_*(\sectiontriviale)$ is a morphism of Batalin-Vilkovisky algebra.

The composite
$$
\Omega G\times M\buildrel{\Omega G\times \sectiontriviale}\over\hookrightarrow \Omega G\times LM
\buildrel{\Phi_{G,M}}\over\rightarrow \Omega G\times LM\buildrel{proj}\over\twoheadrightarrow LM
$$
is $\Theta_{G,M}$. Therefore $H_*(\Theta_{G,M})$ is the composite of the following morphisms of Batalin-Vilkovisky algebras:
\begin{multline*}
H_*(\Omega G)\otimes\mathbb{H}_*(M)\buildrel{H_*(\Omega G)\otimes \mathbb{H}_*(\sectiontriviale)}\over\hookrightarrow H_*(\Omega G)\otimes \mathbb{H}_*(LM)\\
\buildrel{\mathbb{H}_*(\Phi_{G,M})}\over\rightarrow H_*(\Omega G)\otimes \mathbb{H}_*(LM)
\buildrel{\varepsilon_{H_*(\Omega G)}\otimes\mathbb{H}_*(LM)}\over\twoheadrightarrow  \mathbb{H}_*(LM)
\end{multline*}
Note that $\varepsilon_{H_*(\Omega G)}\otimes\mathbb{H}_*(LM)$ is a morphism of Batalin-Vilkovisky algebras
since the augmentation $\varepsilon_{H_*(\Omega G)}:H_*(\Omega G)\twoheadrightarrow{\Bbbk}$
is a morphism of Batalin-Vilkovisky algebras.
So finally, $H_*(\Theta_{G,M})$ is a morphism of Batalin-Vilkovisky algebras.
\end{proof}
\vspace{.3cm}
\begin{definition}~\cite{Hepworth:stringLie}\label{definition homology suspension}
Let $X$ be a pointed space.
Let $\sigma:S^1\times \Omega X\rightarrow X$, $(s,w)\mapsto w(s)$,
denote the evaluation map.
The {\it homology suspension} is the morphism of degree $+1$,
$\sigma_*:H_q(\Omega X)\rightarrow H_{q+1}(X)$ defined
by $\sigma_*(a)=H_*(\sigma)([S^1]\otimes a)$, $a\in H_q(\Omega X)$, $q\geq 0$. 
\end{definition}
According to~\cite[Lemma 7]{Hepworth:stringLie}, this homology suspension coincides with
the usual one studied in~\cite[Chapter VIII]{Whitehead:eltsoht}.
Since this paper was written almost completely before
the preprint of Hepworth appeared, we will never use in this paper this fact
that regretfully, we did not notice.
However, we felt that it was necessary to use his terminology and
we rewrote our paper accordingly.

In~\cite[Theorem 5]{Hepworth:stringLie},
Hepworth computed the Batalin-Vilkovisky algebra on the modulo $2$ free loop space homology on the special
orthogonal group, $\mathbb{H}_*(LSO(n);\mathbb{F}_2)$.
When $n\geq 4$, Lemma 7 of~\cite{Hepworth:stringLie} is required in order to achieve this interesting
computation.
\begin{proposition}\label{operateur BV sur le produit lacets espace}
Let $G$ be a H-group.
Let $X$ be a topological space.
Let $act_X:G\times X\rightarrow X$ be an action up to homotopy of $G$ on $X$.
Suppose that $H_*(\Omega G)$ is torsion free.
Denote by $\BV_{\Omega G}$ (respectively $\BV_{\Omega G\times X}$) the operator given by the action of $H_*(S^1)$
on $H_*(\Omega G)$ (respectively $H_*(\Omega G\times X)$)
given by Proposition~\ref{definition de l'action}.
Then for any $a\in H_*(\Omega G)$, $x\in H_*(X)$,
$$
\BV_{\Omega G\times X}( a\otimes x)
=\left(\BV_{\Omega G} a\right)\otimes x
+\sum (-1)^{\vert a_{(1)}\vert}
a_{(1)}\otimes (\sigma_*(a_{(2)})\cdot x).
$$
Here $\Delta a=\sum a_{(1)}\otimes a_{(2)}$ is the diagonal of
$a\in H_*(\Omega G)$.
\end{proposition}
\begin{proof}
By Proposition~\ref{definition de l'action} ii),
the action of $H_*(S^1)$ on $H_*(\Omega G)\otimes H_*(X)$
is the composite
$$
\xymatrix{
H_*(S^1)\otimes H_*(\Omega G)\otimes H_*(X)
\ar[d]^{\Delta_{H_*(S^1)\otimes H_*(\Omega G)}\otimes H_*(X)}\\
H_*(S^1)\otimes H_*(\Omega G)\otimes H_*(S^1)\otimes H_*(\Omega G)\otimes H_*(X)
\ar[d]^{act_{H_*(\Omega G)}\otimes H_*(\sigma)\otimes H_*(X)}\\
H_*(\Omega G)\otimes H_*(G)\otimes H_*(X)
\ar[d]^{H_*(\Omega G)\otimes act_{H_*(X)}}\\
H_*(\Omega G)\otimes H_*(X)}
$$
\noindent where $\Delta_{H_*(S^1)\otimes H_*(\Omega G)}$ is the diagonal
of $H_*(S^1)\otimes H_*(\Omega G)$ and
$act_{H_*(\Omega G)}:H_*(S^1)\otimes H_*(\Omega G)
\rightarrow H_*(\Omega G)$ is the action of $H_*(S^1)$ on
$H_*(\Omega G)$ given by Proposition~\ref{definition de l'action} i). 

Since \begin{multline*}\Delta_{H_*(S^1)\otimes H_*(\Omega G)}([S^1]\otimes a)= \\
\sum [S^1]\otimes a_{(1)} \otimes 1 \otimes a_{(2)}
+ \sum (-1)^{\vert a_{(1)}\vert}1\otimes a_{(1)} \otimes [S^1] \otimes a_{(2)} ,
\end{multline*}
\begin{multline*}
\BV_{\Omega G\times X}( a\otimes x)
=\sum \left(\BV_{\Omega G} a_{(1)}\right)\otimes H_*(\sigma)(1 \otimes a_{(2)})\cdot x\\
+\sum (-1)^{\vert a_{(1)}\vert}
a_{(1)}\otimes H_*(\sigma)([S^1] \otimes a_{(2)})\cdot x.
\end{multline*}
Let $\varepsilon_{H_*(\Omega G)}:
H_*(\Omega G)\twoheadrightarrow {\Bbbk}$ be the augmentation of the
Hopf algebra $H_*(\Omega G)$.
Since the restriction of $\sigma$ to $\{0\}\times\Omega G$
is the composite $\Omega G\rightarrow \{e\}\rightarrow G$,
$$H_*(\sigma)(1\otimes a)=\varepsilon_{H_*(\Omega G)}(a)1.$$
In a Hopf algebra,  $(Id\otimes\varepsilon)\circ\Delta=Id$.
Therefore,
$$\sum \left(\BV_{\Omega G} a_{(1)}\right)\otimes H_*(\sigma)(1 \otimes a_{(2)})\cdot x=\sum \left(\BV_{\Omega G} a_{(1)}\right)\otimes \varepsilon_{H_*(\Omega G)}( a_{(2)})x= \left(\BV_{\Omega G} a\right)\otimes x.
$$
On the other hand, by Definition~\ref{definition homology suspension}
$$\sum (-1)^{\vert a_{(1)}\vert}
a_{(1)}\otimes H_*(\sigma)([S^1] \otimes a_{(2)})\cdot x=
\sum (-1)^{\vert a_{(1)}\vert}
a_{(1)}\otimes (\sigma_*(a_{(2)})\cdot x).
$$
\end{proof}
\begin{lemma}\label{lacets pointe d'un groupe de Lie}
Let $G$ be a path-connected Lie group.
Then $H_*(\Omega G)$ is ${\Bbbk}$-free and concentrated in even degree.
So for any $a\in H_*(\Omega G)$, $\BV_{\Omega G} a=0$.
\end{lemma}
\begin{proof}
Let $\Omega_0 G$ be the path-component of the constant loop $\hat{e}$ in
$\Omega G$. Since ($\Omega G,\hat{e})$ is a $H$-group,
the composite of the inclusion map and of the multiplication
$$
{\Bbbk}[\pi_0(\Omega G)]\otimes H_*(\Omega_0 G)
\hookrightarrow H_0(\Omega G)\otimes H_*(\Omega G)\rightarrow H_*(\Omega G)
$$
is an isomorphism (of algebras since $H_*(\Omega G)$ is commutative).
Let $\tilde{G}$ be the universal cover of $G$.
Then we have an isomorphism of algebras
$H_*(\Omega\tilde{G})\buildrel{\cong}\over\rightarrow H_*(\Omega_0 G)$.
Since $\tilde{G}$ is a simply-connected Lie group, by a result of Bott~\cite[Theorem 21.7 and Remark 2]{Milnor:Morsetheory},
$H_*(\Omega\tilde{G})$ is  ${\Bbbk}$-free and concentrated in even degree.
Therefore $H_*(\Omega G)\cong {\Bbbk}[\pi_1(G)]\otimes H_*(\Omega\tilde{G})$
is  also ${\Bbbk}$-free and concentrated in even degree.
Therefore $\BV_{\Omega G}:H_*(\Omega G)\rightarrow H_{*+1}(\Omega G)$
is trivial.
\end{proof}
From Remark~\ref{theta homeo dans le cas des groupes de Lie},
Proposition~\ref{operateur BV sur le produit lacets espace}
and Lemma~\ref{lacets pointe d'un groupe de Lie},
since here $(-1)^{\vert a_{(1)}\vert}$ is equal to $1$, we immediately obtain 
 the following Corollary due to Hepworth.
\vspace{.3cm}
\begin{corollary}~\cite[Theorem 1]{Hepworth:stringLie}\label{string homology Lie group}
Let $G$ be a path-connected compact Lie group.
Then as Batalin-Vilkovisky algebra,
$$
\mathbb{H}_*(LG)\cong H_*(\Omega G)\otimes \mathbb{H}_*(G)
$$
and $\BV (a\otimes x)=\sum 
a_{(1)}\otimes (\sigma_*(a_{(2)})\cdot x)$
for $a\in H_*(\Omega G)$, $x\in \mathbb{H}_*(G)$.
\end{corollary}
\begin{corollary}\label{homologie de variete sous BV algebre}
Let $G$ be a $H$-group acting up to homotopy on
a topological space $X$.
Assume that $H_*(\Omega G)$ is torsion free.
Then $H_*(X)$ is a trivial sub $H_*(S^1)$-module of the
$H_*(S^1)$-module $H_*(\Omega G)\otimes H_*(X)$
given by Proposition~\ref{definition de l'action} ii). 
\end{corollary}
Note that the composite
$X\cong \{\hat{e}\}\times X\rightarrow
\Omega G\times X
\buildrel{\Theta_{G,X}}\over\rightarrow LX$
is homotopic to $
\sectiontriviale:X\hookrightarrow LX
$
the trivial section mapping $x\in X$ to $\hat{x}$ the free loop constant on $x$.
Through $\sectiontriviale$, $X$ is a trivial sub $S^1$-space of $LX$.
\begin{proof}[First proof of Corollary~\ref{homologie de variete sous BV algebre}  without using Proposition~\ref{operateur BV sur le produit lacets espace}]
By definition, in Proposition~\ref{definition de l'action} i),
the action of any $t\in S^1$ on the constant loop $\hat{e}\in \Omega (G,e)$
is $\hat{e}$, since $r:LG\rightarrow \Omega G$ is a pointed map.
Therefore by definition, in Proposition~\ref{definition de l'action} ii),
the action of $t\in S^1$ on $(\hat{e},x)\in\Omega G\times X$
is $t.(\hat{e},x)=(t.\hat{e},\hat{e}(s).x)= (\hat{e},e.x)$.
Therefore $X\cong \{\hat{e}\}\times X$ is up to homotopy
a trivial sub $S^1$-space of $\Omega G\times X$.
And in homology, if $H_*(X)$ is considered as a trivial $H_*(S^1)$-module,
the morphism  $H_*(X)\rightarrow H_*(\Omega G)\otimes H_*(X)$,
$x\mapsto 1\otimes x$, is a morphism  of $H_*(S^1)$-modules.
\end{proof}
\begin{proof}[Second proof  of Corollary~\ref{homologie de variete sous BV algebre} using Proposition~\ref{operateur BV sur le produit lacets espace}]
Again, we observe that the constant loop $\hat{e}\in\Omega G$ on the neutral element is  a fixed point under the $S^1$-action of $\Omega G$. Therefore 
$\BV_{\Omega G}(1)=0$.

Let $\varepsilon_{H_*(S^1)}:H_*(S^1)\twoheadrightarrow {\Bbbk}$ be the augmentation. Since the restriction of $\sigma$ to $S^1\times \{\hat{e}\}$
is the composite $S^1\rightarrow  \{e\}\rightarrow G$,
\begin{equation}\label{suspension element neutre}
\sigma_*(1):=H_*(\sigma)([S^1]\otimes 1)=\varepsilon_{H_*(S^1)}([S^1])1=0.
\end{equation}
By Proposition~\ref{operateur BV sur le produit lacets espace},
$$
\BV_{\Omega G\times X}( 1\otimes x)
=\left(\BV_{\Omega G} 1\right)\otimes x
+1\otimes (\sigma_*(1)\cdot x)=0+0.
$$
\end{proof}
\begin{corollary}\label{comparaison delta} 
Let $G$ be a $H$-group acting up to homotopy on
a smooth compact oriented manifold $M$.
Assume that $H_*(\Omega G)$ is torsion free.
Then $H_*(\Omega G)$ is a sub $H_*(S^1)$-module of the
$H_*(S^1)$-module $H_*(\Omega G)\otimes \mathbb{H}_*(M)$
given by Proposition~\ref{definition de l'action} ii). 
\end{corollary}
\begin{proof}[First proof]

For any $a\in H_*(\Omega G)$,
since $\sigma_*(a_{(2)})$ has positive degree, its action on the fundamental class $[M]$ is null.
Therefore 
$ \sum (-1)^{\vert a_{(1)}\vert}   a_{(1)}\otimes \sigma_*(a_{(2)})\cdot [M]
=0 $. So by Proposition~\ref{operateur BV sur le produit lacets espace},
$
\BV_{\Omega G\times M}( a\otimes [M])
=\left(\BV_{\Omega G} a\right)\otimes [M]
+0$.
Therefore,we have proved
$$Id\otimes [M]:H_*(\Omega G)\rightarrow 
H_*(\Omega G)\otimes\mathbb{H}_*(M)$$
is a morphism of $H_*(S^1)$-modules. 
\end{proof}
\begin{proof}[Second proof using shriek maps]
The trivial fibration $proj:\Omega G\times M\rightarrow\Omega G$ is $S^1$-equivariant.
Therefore by~\cite[Section 2.3 Borel Construction]{Chataur-Menichi:stringclass}, the integration along
the fiber of $proj$, $proj_!:H_*(\Omega G)\rightarrow H_*(\Omega G)\otimes\mathbb{H}_*(LM)$
is $H_*(S^1)$-linear. But $proj_!=Id\otimes [M]$.
\end{proof}
\begin{theorem}\label{morphism de BV algebres des lacets doubles vers lacets libres}
Assume the hypothesis of Theorem~\ref{Structure BV sur le produit tensoriel}.
Then the composite
$$H_*(\Omega^2 BG)\cong H_*(\Omega G)
\buildrel{Id\otimes [M]}
\over\rightarrow 
H_*(\Omega G)\otimes\mathbb{H}_*(M)
\buildrel{\mathbb{H}_*(\Theta_{G,M})}\over\rightarrow \mathbb{H}_*(LM)$$
is a morphism of Batalin-Vilkovisky algebras.
\end{theorem}
As pointed in the introduction of this paper, this theorem can be deduce using
Theorem~\ref{Structure BV sur le produit tensoriel}.
But we prefer to give an independent proof.
\begin{proof}[Proof without using~Theorem~\ref{Structure BV sur le produit tensoriel}]
By Proposition~\ref{lacets pointe d'un groupe est BV}, the obvious isomorphism
of algebras between $H_*(\Omega^2 BG)$ and $H_*(\Omega G)$ is an isomorphism of Batalin-Vilkovisky algebra. By Corollary~\ref{comparaison delta}, the inclusion of algebras
$$Id\otimes [M]:H_*(\Omega G)\rightarrow 
H_*(\Omega G)\otimes\mathbb{H}_*(M)$$
is also a morphism of $H_*(S^1)$-modules.
By Theorem~\ref{Felix et Thomas}, $\mathbb{H}_*(\Theta_{G,M})$ is a morphism
of algebras. By Proposition~\ref{definition de l'action} ii),
$\mathbb{H}_*(\Theta_{G,M})$ is also a morphism of $H_*(S^1)$-modules.
\end{proof}
From Theorem~\ref{morphism de BV algebres des lacets doubles vers lacets libres}
and Lemma~\ref{lacets pointe d'un groupe de Lie}, we obtain:
\begin{corollary}
Let $G$ be a path-connected compact Lie group. Then the algebra
$H_*(\Omega^2 BG)$ equipped with the operator $\BV=0$ is a sub Batalin-Vilkovisky algebra of
$\mathbb{H}_*(LG)$.
\end{corollary}
\begin{corollary}\label{expression simple du crochet}
Assume the hypothesis of Theorem~\ref{Structure BV sur le produit tensoriel}.
Let $a\in H_*(\Omega G)$ and $x\in\mathbb{H}_*(M)$.
Then the Lie bracket of $a\otimes [M]$ and $1\otimes x$ in the Batalin-Vilkovisky
algebra $H_*(\Omega G)\otimes\mathbb{H}_*(M)$ 
is given by 
$$\{ a\otimes [M],1\otimes x\}=\sum (-1)^{\vert a_{(2)}\vert}
a_{(1)}\otimes \sigma_*(a_{(2)})\cdot x.$$
In particular, if $a$ is primitive,
$\{ a\otimes [M],1\otimes x\}=(-1)^{\vert a\vert}
1\otimes \sigma_*(a)\cdot x$.
\end{corollary}
\begin{proof}
Since by Corollary~\ref{comparaison delta} and Corollary~\ref{homologie de variete sous BV algebre},
$H_*(\Omega G)$ and $\mathbb{H}_*(M)$  are two sub-Batalin-Vilkovisky algebras
of the Batalin-Vilkovisky algebra  $H_*(\Omega G)\otimes\mathbb{H}_*(M)$,
\begin{equation}\label{calcul du crochet tordu}
\{ a\otimes [M],1\otimes x\}=
(-1)^{\vert a\vert} \BV_{\Omega G\times M}\left(a\otimes x\right)
-(-1)^{\vert a\vert} \left(\BV_{\Omega G} a\right)\otimes x
-a\times 0.
\end{equation}
Therefore, by Proposition~\ref{operateur BV sur le produit lacets espace},
$$\{ a\otimes [M],1\otimes x\}=\sum (-1)^{\vert a_{(2)}\vert}
a_{(1)}\otimes \sigma_*(a_{(2)})\cdot x.$$
In particular if $\Delta a=a\otimes 1+1\otimes a$,
$$\{ a\otimes [M],1\otimes x\}=a\otimes \sigma_*(1).x+
(-1)^{\vert a\vert} 1\otimes \sigma_*(a)\cdot x.$$
By equation~(\ref{suspension element neutre}), $\sigma_*(1)=0$.
Therefore $\{ a\otimes [M],1\otimes x\}=0+
(-1)^{\vert a\vert} 1\otimes \sigma_*(a)\cdot x$.
\end{proof}
\section{some computations}
Using Hepworth's definition of the homology suspension $\sigma_*$
(Definition~\ref{definition homology suspension}), Lemma 11
of~\cite{menichi:stringtopspheres} becomes the well-known fact.
\begin{lemma}\label{suspension, hurewicz et adjonction}
Let $X$ be a pointed topological space.
Let $n\geq 0$.
Denote by $hur_X:\pi_n(X)\rightarrow H_n(X)$ the Hurewicz map.
We have the commutative diagram
$$
\xymatrix{
\pi_n(\Omega X)\ar[r]^\cong\ar[d]_{hur_{\Omega X}}
&\pi_{n+1}(X)\ar[d]^{hur_X}\\
H_n(\Omega X)\ar[r]_{\sigma_*}
&H_{n+1}(X)
}$$
where $\pi_n(\Omega X)\cong \pi_{n+1}(X)$ is the adjunction map.
\end{lemma}
In~\cite[Theorem 10]{menichi:stringtopspheres}
and then in~\cite[Proposition 9]{Hepworth:stringLie},
this Lemma was used to computed the Batalin-Vilkovisky algebra
$\mathbb{H}_*(LS^1)$.
The same proof shows the following Proposition.
\begin{proposition}\label{la BV-algebre dans le cas du cercle}
Let $M$ be a compact oriented manifold equipped with
an action of the circle $S^1$.
Denote by $x$ a generator of $\mathbb{Z}$.
Then the Batalin-Vilkovisky algebra
$H_*(\Omega S^1)\otimes \mathbb{H}_*(M)$
is the tensor product of graded algebras
${\Bbbk}[\mathbb{Z}]\otimes \mathbb{H}_*(M)$
with 
$\BV_{\Omega S^1\times M}( x^i\otimes m)= i x^i\otimes\left([S^1].m\right)$
for any $i\in\mathbb{Z}$, $m\in \mathbb{H}_*(M)$.
\end{proposition}
\begin{proof}
By applying Lemma~\ref{suspension, hurewicz et adjonction},
to the degree $i$ map $S^1\rightarrow S^1$,
we obtain that $\sigma_*(x^i)=i[S^1]$.
By applying Proposition~\ref{operateur BV sur le produit lacets espace}
and Lemma~\ref{lacets pointe d'un groupe de Lie},
we conclude.
\end{proof}
The following Proposition generalises the computation of the Batalin-Vilkovisky
algebra $\mathbb{H}_*(LS^3)$ due independently to the
author~\cite[Theorem 16]{menichi:stringtopspheres}
and to Tamanoi~\cite{Tamanoi:BVLalg}.
\begin{proposition}\label{calcul pour sphere S3}
Let $M$ be a compact oriented manifold equipped with
an action of the sphere $S^3$.
Then the Batalin-Vilkovisky algebra
$H_*(\Omega S^3)\otimes \mathbb{H}_*(M)$
is the tensor product of graded algebras
${\Bbbk}[u_2]\otimes \mathbb{H}_*(M)$
with 
$\BV_{\Omega S^1\times M}( u_2^i\otimes m)= i u_2^{i-1}\otimes\left([S^3].m\right)$
for any $i\in\mathbb{N}$, $m\in \mathbb{H}_*(M)$.
\end{proposition}
\begin{proof}
Let $ad:S^2\rightarrow \Omega S^3$ be the adjoint map and let $u_2$
be $H_*(ad)([S^2])$, the image of $[S^2]$ by  $H_*(ad)$.
By Bott-Samelson theorem, $H_*(\Omega S^3)$ is the polynomial algebra
${\Bbbk}[u_2]$. By Lemma~\ref{suspension, hurewicz et adjonction},
we obtain that $\sigma_*(u_2)=[S^3]$.
Therefore since $u_2$ is primitive,
by Corollary~\ref{expression simple du crochet},
$\{ u_2\otimes [M],1\otimes m \} =1\otimes\left([S^3].m\right)$.
By Corollary~\ref{homologie de variete sous BV algebre},
$\BV_{\Omega S^3\times M}(1\otimes m)=0$.
By Corollary~\ref{comparaison delta},
$\BV_{\Omega S^3\times M}(u_2^i\otimes  [M])=
\BV_{\Omega S^3}(u_2^i)\otimes  [M]=0$.
Using the Poisson
relation~(\ref{Poisson relation}), 
$\{u_{2}^i\otimes [M],1\otimes m\}=iu_{2}^{i-1}\otimes\left([S^3].m\right)$.
Therefore $\BV(u_{2}^i\otimes m)=i u_{2}^{i-1}\otimes\left([S^3].m\right).$
\end{proof}
In~\cite[Proposition 10]{Hepworth:stringLie},
Hepworth gave a new proof for the computation of the Batalin-Vilkovisky
algebra $\mathbb{H}_*(LS^3)$.
This proof of Hepworth gives also a proof of Proposition~\ref{calcul pour sphere S3}.
\section{A sub Lie algebra of $\mathbb{H}_*(LM;\mathbb{Q})$}
Let $M$ be a smooth compact oriented manifold.
We show that the rational free loop homology on $M$,
$\mathbb{H}_*(LM;\mathbb{Q})$, equipped with the loop bracket contains
a sub Lie algebra. 
\begin{definition}(\cite[p. 272]{Gerstenhaber:cohosar}, \cite[7.4.9]{Weibel:inthomalg}, \cite[p. 235]{Hilton-Stammbach})\label{semidirect de Lie}
Let $L$ be a graded Lie algebra. Let $N$ be a left $L$-module.
Then the direct product of graded ${\Bbbk}$-modules $N\times L$ can be equipped with
a Lie bracket defined by
$$
\{(m,l),(m',l')\}:=(l.m'-(-1)^{\vert l'\vert \vert m\vert} l'.m,\{l,l'\})
$$
for $m$, $m'\in N$ and $l,l'\in L$.
This graded Lie algebra is denoted $N\rtimes L$ and is called the {\it semidirect product} of $N$ and $L$
or {\it trivial extension} of $L$ by $N$.
\end{definition}
\begin{example}(The Lie bracket of degree $+1$ on $s^{-1}\pi_{\geq 2}(G)\otimes {\Bbbk}\oplus \mathbb{H}_*(M)$)\label{exemple produit semidirect Lie}
Let $G$ be a path-connected topological monoid
acting on a smooth compact oriented manifold $M$.
The associative algebra $H_*(G)$ acts on
$\mathbb{H}_*(M)$ and therefore on $s\mathbb{H}_*(M)$ by
$
a.sx:=(-1)^{\vert a\vert -1}s(a.x)
$
for $a\in H_*(G)$ and $x\in\mathbb{H}_*(M)$.
Since $G$ is a path-connected topological monoid, by~\cite[X.6.3]{Whitehead:eltsoht},
the Hurewicz morphism $\textit{hur}_G:\pi_*(G)\rightarrow H_*(G)$
is a morphism of graded Lie algebras from the Samelson product to the commutator associated to
the associative algebra $H_*(G)$.
The truncated homotopy groups of $G$, $\pi_{\geq 2}(G)$ is a sub-Lie algebra
of $\pi_*(G)$. So finally, the composite
$$
\pi_{\geq 2}(G)\hookrightarrow \pi_*(G)\buildrel{\textit{hur}_G}\over\rightarrow
H_*(G)\rightarrow\text{Hom}(s\mathbb{H}_*(M),s\mathbb{H}_*(M))
$$
is a morphism of graded Lie algebras, i.e.~$s\mathbb{H}_*(M)$
is a module over the Lie algebra $\pi_{\geq 2}(G)$.
Therefore by the definition of the semidirect product
$s\mathbb{H}_*(M)\rtimes \pi_{\geq 2}(G)$ (Definition~\ref{semidirect de Lie} above),
$s\mathbb{H}_*(M)\oplus \pi_{\geq 2}(G)$ is a graded Lie algebra.
By desuspending, $\mathbb{H}_*(M)\oplus s^{-1}\pi_{\geq 2}(G)$
has a Lie bracket of degree $+1$.

Explicitly the bracket on  $s^{-1}\pi_{\geq 2}(G)\otimes {\Bbbk}\oplus \mathbb{H}_*(M)$
is defined by 

$\bullet$ $\{s^{-1}f,s^{-1}g\}:=s^{-1}\{f,g\}$, (here $\{f,g\}$ is the Samelson bracket of $f$ and $g\in\pi_{\geq 2}(G)\otimes {\Bbbk}$)

$\bullet$ $\{x,y\}:=0$ for $x$ and $y\in  \mathbb{H}_*(M)$, 

$\bullet$ $\{s^{-1}f,x\}:=(-1)^{\vert f\vert -1}\mathit{hur}_Gf.x$. (Here $\mathit{hur}_G:\pi_{*}(G)\otimes {\Bbbk}\rightarrow H_*(G)$ is the Hurewicz morphism.)
\end{example}
\begin{theorem}\label{trivial extension vers loop bracket}
Let $G$ be a path-connected topological monoid with an homotopy inverse
acting on a smooth compact oriented manifold $M$.
Assume that $H_*(\Omega G)$ is torsion free.
Let $\tilde{\Gamma}_1:s^{-1}\pi_{\geq 2}(G)\otimes {\Bbbk}\rightarrow \mathbb{H}_*(LM)$ be the composite
of the adjunction map, the Hurewicz morphism and of the map considered in
Theorem~\ref{morphism de BV algebres des lacets doubles vers lacets libres}
$$\pi_{n+1}G\otimes{\Bbbk}\cong \pi_{n}\Omega G\otimes {\Bbbk}\buildrel{\mathit{hur}_{\Omega G}}\over\rightarrow H_n(\Omega G)
\buildrel{Id\otimes [M]}
\over\rightarrow 
H_n(\Omega G)\otimes\mathbb{H}_0(M)
\buildrel{\mathbb{H}_n(\Theta_{G,M})}\over\rightarrow \mathbb{H}_n(LM)$$
for $n\geq 1$. 

Then the ${\Bbbk}$-linear morphism
$s^{-1}\pi_{\geq 2}(G)\otimes {\Bbbk}\oplus \mathbb{H}_*(M)\rightarrow \mathbb{H}_*(LM)$ mapping $(s^{-1}f,x)$ to $\tilde{\Gamma}_1(s^{-1}f)+\mathbb{H}_*(s)(x)$
is a morphism of Lie algebras between the Lie bracket from example~\ref{exemple produit semidirect Lie} and the loop bracket of $LM$.
\end{theorem}
Recall that $\sectiontriviale:M\hookrightarrow LM$
denotes the trivial section.
\begin{proof}
Recall from the proof of Proposition~\ref{lacets pointe d'un groupe est BV},
that there exists a homomorphism of $H$-spaces $h:G\buildrel{\simeq}\over\rightarrow
\Omega BG$ which is also a weak equivalence and that
$H_*(\Omega h):H_*(\Omega G)\buildrel{\cong}\over\rightarrow
H_*(\Omega^2 BG)$ is an isomorphism of Batalin-Vilkovisky algebras.
Since $h:G\buildrel{\simeq}\over\rightarrow\Omega BG$
is a homomorphism of $H$-spaces between two path-connected homotopy associative
$H$-spaces, $\pi_*(h):\pi_*(G)\buildrel{\cong}\over\rightarrow
\pi_*(\Omega BG)$ is a morphism of graded Lie algebras with respect to the
Samelson brackets on $G$ and on $\Omega BG$.
For $n\geq 1$, consider the commutative diagram of graded ${\Bbbk}$-modules
$$
\xymatrix{
\pi_{n+1}(\Omega BG)\otimes {\Bbbk}\ar[r]^\cong
& \pi_{n}(\Omega^2 BG)\otimes {\Bbbk}\ar[r]^{\textit{hur}_{\Omega^2 BG}}
& H_{n}(\Omega^2 BG)
\\
\pi_{n+1}(G)\otimes {\Bbbk}\ar[r]^\cong\ar[u]^{\pi_{n+1}(h)\otimes {\Bbbk}}
& \pi_{n}(\Omega G)\otimes {\Bbbk}\ar[r]^{\textit{hur}_{\Omega G}}
\ar[u]|{\pi_{n}(\Omega h)\otimes {\Bbbk}}
& H_{n}(\Omega G)\ar[u]_{H_{n}(\Omega h)}
}
$$
By~\cite[Remark 1.2 p. 214-5]{Cohen-Lada-May:homiterloopspaces},
the top line is a morphim of Lie algebras between the Samelson bracket and the Browder bracket.
Therefore, the bottom line is also a morphism of Lie algebras.
By Theorem~\ref{morphism de BV algebres des lacets doubles vers lacets libres},
the composite $$H_n(\Omega G)
\buildrel{Id\otimes [M]}
\over\rightarrow 
H_n(\Omega G)\otimes\mathbb{H}_0(M)
\buildrel{\mathbb{H}_n(\Theta_{G,M})}\over\rightarrow \mathbb{H}_n(LM)$$
is a morphism of Lie algebras.
Therefore, by composition,
$$\tilde{\Gamma}_1:s^{-1}\pi_{n+1}(G)\otimes {\Bbbk}\rightarrow \mathbb{H}_n(LM)$$
is a morphism of Lie algebras between the Samelson bracket of $G$ and the loop bracket of $LM$.

The commutative graded algebra $\mathbb{H}_*(M)$ equipped with the trivial operator $\BV$
can be considered as a Batalin-Vilkovisky algebra.
Since the trivial section $\sectiontriviale:M\hookrightarrow LM$ is $S^1$-equivariant with respect
to the trivial action on $M$,
$\mathbb{H}_*(\sectiontriviale):\mathbb{H}_*(M)\hookrightarrow \mathbb{H}_*(LM)$ is an inclusion of
Batalin-Vilkovisky algebras and so an inclusion of Lie algebras.

Denote by $\tilde{f}\in\pi_n(\Omega G)$, the adjoint of $f\in\pi_{n+1}(G)$.
By Lemma~\ref{suspension, hurewicz et adjonction}, $\sigma_*\circ \textit{hur}_{\Omega G}\tilde{f}=\textit{hur}_{G}f$.
Since $n\geq 1$, $a:=\textit{hur}_{\Omega G}\tilde{f}$ is primitive and so by
Corollary~\ref{expression simple du crochet},
$\{\textit{hur}_{\Omega G}\tilde{f}\otimes [M], 1\otimes x\}=
(-1)^{n}1\otimes\textit{hur}_{G}f.x$.
Since by Theorem~\ref{Structure BV sur le produit tensoriel},
$
H_*(\Theta_{G,M}):H_*(\Omega G)\otimes\mathbb{H}_*(M)\rightarrow \mathbb{H}_*(LM)
$
is a morphism of Lie algebras, finally we have
$$\{\tilde{\Gamma_1}(f), \mathbb{H}_*(\sectiontriviale)(x)\}=(-1)^{n}\mathbb{H}_*(\sectiontriviale)(\textit{hur}_{G}f.x).$$
\end{proof}
\begin{theorem}\label{rationellement injection de Lie}
If ${\Bbbk}$ is a field of characteristic $0$ and $G$ is $aut_1M$,
the monoid of self-equivalences homotopic to the identity,
the morphism of Lie algebras considered in
Theorem~\ref{trivial extension vers loop bracket}
is injective.
\end{theorem}
\begin{proof}
Felix and Thomas~\cite[Theorem 2]{Felix-Thomas:monsefls} showed that for $n\geq 1$,
$\tilde{\Gamma_1}:\pi_{n+1}(aut_1M)\otimes{\Bbbk}\hookrightarrow \mathbb{H}_n(LM)$
is injective.
Since $\sectiontriviale$ is a section,
$\mathbb{H}_*(\sectiontriviale)=\mathbb{H}_*(M)\hookrightarrow \mathbb{H}_*(LM)$ is also injective.
Our morphism of Lie algebras coincides with $\tilde{\Gamma_1}$ in positive degree  and with
$\mathbb{H}_*(\sectiontriviale)$ in non-positive degree. Therefore, we have proved the theorem.
\end{proof}
Here is a example due to Yves Felix showing that the Lie algebras considered
in Theorem~\ref{rationellement injection de Lie}
are not abelian even for a very simple manifold $M$.
\begin{example}
Let $M$ be the product of spheres $S^3\times S^7$.
The minimal Sullivan model of $M$ is $(\Lambda(x_3,y_7),0)$ with $x_3$ in (upper) degree $3$
and $y_7$ in (upper) degree $7$.
Consider $\text{Der}\,\Lambda(x_3,y_7)$ the Lie algebra of derivation of $\Lambda(x_3,y_7)$
decreasing the degree.
Let $\theta_1$,  $\theta_2$ and $\theta_3\in \text{Der}\,\Lambda(x_3,y_7)$ given
by $\theta_1(x_3)=1$, $\theta_1(y_7)=0$, $\theta_2(x_3)=0$, $\theta_2(y_7)=1$,
$\theta_3(x_3)=0$ and $\theta_3(y_7)=x_4$. These three derivations form a
basis of the graded vector space $\text{Der}\,\Lambda(x_3,y_7)$.
As graded Lie algebras, $\pi_{\geq 2}(aut_1 M)\otimes\mathbb{Q}$ is isomorphic to $\text{Der}\,\Lambda(x_3,y_7)$.
Since $\theta_2=\theta_1\circ\theta_3=\{\theta_1,\theta_3\}$, the Samelson bracket in   $\pi_{\geq 2}(aut_1 M)\otimes\mathbb{Q}$ is non trivial.
\end{example}
\section{The Batalin-Vilkovisky algebra $\mathbb{H}_*(LG;\mathbb{Q})$}
\begin{theorem}\label{calcul rationel de la BV-algebre}
Let $G$ be a path-connected topological monoid with a homotopy inverse
acting on a smooth compact oriented manifold $M$.
Then the Batalin-Vilkovisky algebra
$H_*(\Omega G;\mathbb{Q})\otimes \mathbb{H}_*(M;\mathbb{Q})$
is the tensor product
$\mathbb{Q}[\pi_1(G)]\otimes \Lambda(s^{-1}\pi_{\geq 2}(G)\otimes\mathbb{Q})
\otimes \mathbb{H}_*(M;\mathbb{Q})$
of the group ring on $\pi_1(G)$, of the free commutative graded algebra
on $\pi_{\geq 1}(\Omega G)\otimes \mathbb{Q}\cong s^{-1}\pi_{\geq 2}(G)\otimes\mathbb{Q}$
and of $\mathbb{H}_*(M;\mathbb{Q})$ equipped with the intersection product.
For any $f\in \pi_1(G)$, $f_1$,..,$f_r\in \pi_{\geq 2}(G)\otimes\mathbb{Q}$,
$r\geq 0$ and $x\in \mathbb{H}_*(M;\mathbb{Q})$,
\begin{multline}\label{operateur BV dans le cas rationnel}
\BV_{\Omega G\times M}(fs^{-1}f_1\dots s^{-1}f_r\otimes x)
=\BV_{\Omega G}(fs^{-1}f_1\dots s^{-1}f_r)\otimes x\\
+(-1)^{\vert f_{1}\vert+\dots+\vert f_{r}\vert +r}
 (fs^{-1}f_1\dots s^{-1}f_r\otimes \text{hur f}\cdot x+\\
\sum_{i=1}^r
(-1)^{(\vert f_i\vert+1)(\vert f_{i+1}\vert+\dots+\vert f_{r}\vert +r-i+1)}
fs^{-1}f_1\dots s^{-1}f_{i-1}s^{-1}f_{i+1}\dots s^{-1}f_r\otimes \text{hur} f_i\cdot x).
\end{multline}
\end{theorem}
\begin{proof}
The isomorphism of algebras
$H_*(\Omega G)\cong \mathbb{Q}[\pi_1(G)]\otimes H_*(\Omega \tilde{G})$
given in the proof of Lemma~\ref{lacets pointe d'un groupe de Lie}
is in fact an isomorphism of Hopf algebras.
Therefore $f$ is a group-like element of $H_*(\Omega G)$.
So by Corollary~\ref{expression simple du crochet}
$\{f\otimes [M], 1\otimes x\}
=f\otimes hur_G f\cdot x$.

By Milnor-Moore~\cite[Theorem 21.5]{Felix-Halperin-Thomas:ratht}
and Cartan-Serre theorems,
the Hurewicz morphism 
$\pi_{\geq 1}\Omega G\otimes \mathbb{Q}\cong
\pi_{\geq 1}\Omega \tilde{G}\otimes \mathbb{Q}
\rightarrow H_*(\Omega \tilde{G};\mathbb{Q})$
extends to an isomorphism of Hopf algebras
$\Lambda (\pi_{\geq 1}\Omega G\otimes \mathbb{Q})\buildrel{\cong}\over\rightarrow
H_*(\Omega \tilde{G};\mathbb{Q})$.
Denote by $s^{-1}f_i\in\pi_{\vert f_i\vert-1}(\Omega G)\otimes \mathbb{Q}$
the adjoint of $f_i\in\pi_{\vert f_i\vert}(G)\otimes \mathbb{Q}$.
As we already saw in Theorem~\ref{trivial extension vers loop bracket},
by Corollary~\ref{expression simple du crochet} and
Lemma~\ref{suspension, hurewicz et adjonction},
$\{s^{-1}f_i\otimes [M], 1\otimes x\}
=(-1)^{\vert s^{-1}f_i\vert}1\otimes hur_G f_i\cdot x$.
Using the graded commutativity of the product and the graded antisymmetry
of the Lie bracket of degree $+1$, the Poisson relation~(\ref{Poisson relation})
can be rewritten as
$$
\{ bc,a \}=b\{ c,a\}+(-1)^{\vert b\vert \vert c\vert}c\{ b,a\}
$$
in any Gerstenhaber algebra. 
Note that the sign in this formula is given exactly by the Koszul rule.
Therefore by immediate induction,
\begin{multline*}
\{ fs^{-1}f_1\dots s^{-1}f_r\otimes [M], 1\otimes x\}=
\{ f\otimes [M]. s^{-1}f_1\otimes [M]\dots s^{-1}f_r\otimes [M],1\otimes x\}\\
=s^{-1}f_1\dots s^{-1}f_r \otimes [M].\{f\otimes [M], 1\otimes x\}\\
+\sum_{i=1}^r (-1)^{\vert s^{-1}f_i\vert\vert s^{-1}f_{i+1}\dots s^{-1}f_{r}\vert}
fs^{-1}f_1\dots s^{-1}f_{i-1}s^{-1}f_{i+1}\dots s^{-1}f_r\otimes [M]
.\{s^{-1}f_i\otimes [M],1\otimes x\}\\
=f s^{-1}f_1\dots s^{-1}f_r \otimes hur_G f.x\\
+\sum_{i=1}^r (-1)^{\vert s^{-1}f_i\vert\vert s^{-1}f_{i+1}\dots s^{-1}f_{r}\vert+\vert s^{-1}f_i\vert}
fs^{-1}f_1\dots s^{-1}f_{i-1}s^{-1}f_{i+1}\dots s^{-1}f_r\otimes hur_G f_i.x
\end{multline*}
Let $a=fs^{-1}f_1\dots s^{-1}f_r \in H_*(\Omega G)$.
By equation~(\ref{calcul du crochet tordu})
$$
\BV_{\Omega G\times M}(a\otimes x)=
\BV_{\Omega G}(a)\otimes x
+(-1)^{\vert a\vert}\{a\otimes [M],1\otimes x\}
$$
Therefore the theorem is proved.
\end{proof}
In the proof of Theorem~\ref{calcul rationel de la BV-algebre},
we have only used that the algebra $H_*(\Omega G;\mathbb{Q})$
is generated by its spherical elements (i.e.~elements in the image
of the Hurewicz map).
Therefore the formula~(\ref{operateur BV dans le cas rationnel})
holds over any principal ideal domain ${\Bbbk}$
where $H_*(\Omega G;{\Bbbk})$ is generated by its spherical elements.
This is also the case for example if $G$ is the circle $S^1$ or the
three-dimensional sphere $S^3$. This explains why formula
~(\ref{operateur BV dans le cas rationnel})
generalizes Propositions~\ref{la BV-algebre dans le cas du cercle}
and~\ref{calcul pour sphere S3}.

Theorem~\ref{calcul rationel de la BV-algebre} tells us in particular
that if we know the Batalin-Vilkovisky algebra $H_*(\Omega G;\mathbb{Q})$ and the action of the spherical elements of $H_*(G)$ on $\mathbb{H}_*(M)$,
we can compute the Batalin-Vilkovisky algebra
$H_*(\Omega G;\mathbb{Q})\otimes\mathbb{H}_*(M;\mathbb{Q})$.
In~\cite[Theorem 4.4]{Gaudens-Menichi}, Gerald Gaudens and the author
computed the Batalin-Vilkovisky algebra
$H_*(\Omega G;\mathbb{Q})\cong H_*(\Omega^2 BG;\mathbb{Q})$
assuming that $G$ is simply-connected.
As we have already seen in Lemma~\ref{lacets pointe d'un groupe de Lie},
the Batalin-Vilkovisky algebra $H_*(\Omega G)$ is also known when
$G$ is a path-connected Lie group.
Therefore, we have:
\begin{corollary}\label{calcul rationel pour les groupes de Lie}
Let $G$ be a path-connected Lie  group
acting on a smooth compact oriented manifold $M$.
Then the Batalin-Vilkovisky algebra
$H_*(\Omega G;\mathbb{Q})\otimes \mathbb{H}_*(M;\mathbb{Q})$
is the tensor product
$\mathbb{Q}[\pi_1(G)]\otimes \Lambda(s^{-1}\pi_{\geq 3}(G)\otimes\mathbb{Q})
\otimes \mathbb{H}_*(M;\mathbb{Q})$
of the group ring on $\pi_1(G)$, of the symmetric algebra
on $ s^{-1}\pi_{\geq 3}(G)$
and of $\mathbb{H}_*(M;\mathbb{Q})$ equipped with the intersection product.
For any $f\in \pi_1(G)$, $f_1$,..,$f_r\in \pi_{\geq 3}(G)\otimes\mathbb{Q}$,
$r\geq 0$ and $x\in \mathbb{H}_*(M;\mathbb{Q})$,
\begin{multline*}
\BV_{\Omega G\times M}(fs^{-1}f_1\dots s^{-1}f_r\otimes x)
=
 fs^{-1}f_1\dots s^{-1}f_r\otimes \text{hur f}\cdot x\\
+\sum_{i=1}^r
fs^{-1}f_1\dots s^{-1}f_{i-1}s^{-1}f_{i+1}\dots s^{-1}f_r\otimes \text{hur} f_i\cdot x.
\end{multline*}
\end{corollary}
Note that there are no signs in this Corollary.
\begin{proof}
The rational homotopy groups $\pi_*(G)\otimes \mathbb{Q}$ are
concentrated in odd degrees. Therefore $s^{-1}f_i$ are all in even degree
and the Corollary follows from Theorem~\ref{calcul rationel de la BV-algebre}.
\end{proof}
\begin{theorem}\label{homologie rationel des lacets sur les groupes de Lie}
Let $G$ be a path-connected compact Lie group.
Let $x_1$, \dots, $x_l$ be a basis of the free part of $\pi_1(G)$.
Let $x_{l+1}$, \dots, $x_r$ be a basis of the $\mathbb{Q}$-vector space
$\pi_{\geq 3}(G)\otimes\mathbb{Q}$.
Denote by $(x_i^\vee)_{1\leq i\leq r}$ the dual basis in
$(\pi_*(G)\otimes\mathbb{Q})^\vee$.
Then the Batalin-Vilkovisky algebra $\mathbb{H}_*(LG;\mathbb{Q})$
is isomorphic to the tensor product
$
\mathbb{Q}[\pi_1(G)]\otimes \Lambda (s^{-1}x_j)_{l<j\leq r}\otimes
\Lambda(x_i^\vee)_{1\leq i\leq r}
$
of the group ring on $\pi_1(G)$, the polynomial algebra on $s^{-1}x_{l+1}$,\dots,
$s^{-1}x_r$ and the exterior algebra on $x_1^\vee$,\dots, $x_r^\vee$.
For $n_1$, \dots,$n_l\in\mathbb{Z}$ and $n_{l+1}$, \dots, $n_{r}\geq 0$, $p\geq 0$ and $1\leq j_1<\dots< j_p\leq r$ and $y$ any element of $\pi_1(G)$ with torsion,
$$
\BV (x_1^{n_1}\dots x_l^{n_l}ys^{-1}x_{l+1}^{n_{l+1}}\dots s^{-1}x_r^{n_r}
\otimes x_{j_1}^\vee\dots x_{j_p}^\vee)$$
$$
=\sum_{i=1, j_i\leq l}^p (-1)^{i-1}n_{j_i} x_1^{n_1}\dots x_l^{n_l}ys^{-1}x_{l+1}^{n_{l+1}}\dots s^{-1}x_r^{n_r}
\otimes x_{j_1}^\vee\dots\widehat{x_{j_i}^\vee}\dots x_{j_p}^\vee $$
$$+\sum_{i=1, j_i > l}^p (-1)^{i-1}n_{j_i} x_1^{n_1}\dots x_l^{n_l}y
s^{-1}x_{l+1}^{n_{l+1}}\dots s^{-1}x_{j_i}^{n_{j_i}-1}\dots s^{-1}x_r^{n_r}
\otimes x_{j_1}^\vee\dots\widehat{x_{j_i}^\vee}\dots x_{j_p}^\vee 
$$
Here $\widehat{\quad}$ denotes omission.
\end{theorem}
When $G=SO(n)$, $n\geq 3$, the special orthogonal group,
the Batalin-Vilkovisky algebra $\mathbb{H}_*(LSO(n);\mathbb{Q})$
has been first computed by Hepworth~\cite[Theorem 4]{Hepworth:stringLie}.
but his formula uses a different presentation of the algebra
$H_*(\Omega SO(n);\mathbb{Q})$ by generators and relations.
\begin{proof}
As we already explained in Remark~\ref{theta homeo dans le cas des groupes de Lie},
$$\mathbb{H}_*(\Theta_{G,G}):H_*(\Omega G)\otimes\mathbb{H}_*(G)
\buildrel{\cong}\over\rightarrow \mathbb{H}_*(LG)$$
is an isomorphism of Batalin-Vilkovisky algebras.
By Milnor-Moore theorem, as Hopf algebras,
$H_*(G)$ is the exterior algebra $\Lambda(x_i)_{1\leq i\leq r}$
where $x_i$ are primitive elements of odd degree.
Therefore $H^*(G)$ is the exterior algebra 
$\Lambda (x_i^\vee)_{1\leq i\leq r}$.

By Poincar\'e duality, the cap product with the fundamental class
$[G]$ gives an isomorphism of graded algebras
$-\cap [G]:H^{*}(G)\buildrel{\cong}\over\rightarrow \mathbb{H}_*(G)$
between the cup product and the intersection product.
(Note that this isomorphism respects degrees since a non-negative upper
degree corresponds to a non-positive lower degree by the classical convention
of~\cite[p. 41-2]{Felix-Halperin-Thomas:ratht}).

The fundamental class of $G$, $[G]$, is the product $x_1\dots x_r$.
Let $i_1$,\dots,$i_p$ be $p$ integers between $1$ and $r$.
Let $1\leq k\leq p$.
The cap product of $x_{i_k}^\vee$ with $x_{i_1}\dots x_{i_p}$
is $(-1)^{k-1} x_{i_k}^\vee(x_{i_k})x_{i_1}\dots\widehat{x_{i_k}}\dots x_{i_p}$
since $(-1)^{k-1}$ is the Koszul sign obtained by exchanging
$x_{i_k}$ and $x_{i_1}\dots x_{i_{k-1}}$.
Therefore $$x_{i_k}^\vee\cap x_{i_1}\dots x_{i_p}=(-1)^{k-1}x_{i_1}\dots\widehat{x_{i_k}}\dots x_{i_p}.$$
In particular,
\begin{align*}
x_{j_p}^\vee\cap x_1\dots x_r
&=(-1)^{j_p-1} x_{1}\dots\widehat{x_{j_p}}\dots x_{r},\\
x_{j_{p-1}}^\vee\cap x_1\dots\widehat{x_{j_p}}\dots x_r
&=(-1)^{j_{p-1}-1}x_{1}\dots\widehat{x_{j_{p-1}}}\widehat{x_{j_p}}\dots x_{r},\\
&.\\
&.\\
\text{and}\;x_{j_{1}}^\vee\cap x_1\dots\widehat{x_{j_2}}\dots\widehat{x_{j_p}}\dots x_r
&=(-1)^{j_1-1} x_{1}\dots\widehat{x_{j_{1}}}\dots\widehat{x_{j_p}}\dots x_{r}.
\end{align*}
Therefore
\begin{multline*}
x_{j_1}^\vee\dots x_{j_p}^\vee\cap [G]=
x_{j_1}^\vee\dots x_{j_p}^\vee\cap x_1\dots x_r=
x_{j_1}^\vee\cap(x_{j_2}^\vee\cap\dots (x_{j_p}^\vee\cap x_1\dots x_r))\\=
(-1)^{j_1-1+\dots +j_p-1} x_{1}\dots\widehat{x_{j_{1}}}\dots\widehat{x_{j_p}}\dots x_{r}.
\end{multline*}
Using the Poincar\'e duality isomorphism, we can now transport the action
of $H_*(G)$ on $\mathbb{H}_*(G)$ given by multiplication into an action
of $H_*(G)$ on $H^*(G)$.
Since
\begin{multline*}
x_{j_i}\cdot\left( x_{j_1}^\vee\dots x_{j_p}^\vee\cap [G]\right)=
(-1)^{j_1-1+\dots j_p-1} x_{j_i}x_{1}\dots\widehat{x_{j_{1}}}\dots\widehat{x_{j_p}}\dots x_{r}\\
=(-1)^{j_i-1-i+1}(-1)^{j_1-1+\dots +j_p-1}
x_1\dots\widehat{x_{j_1}}\dots \widehat{x_{j_{i-1}}}x_{j_i}
\widehat{x_{j_{i+1}}}\dots \widehat{x_{j_p}}\dots x_r\\
=(-1)^{i-1} x_{j_1}^\vee\dots \widehat{x_{j_i}^\vee}\dots x_{j_p}^\vee\cap [G],
\end{multline*}
we obtain that
\begin{equation}\label{action de homologie de G sur la cohomologie de G}
x_{j_i}\cdot x_{j_1}^\vee\dots x_{j_p}^\vee
=(-1)^{i-1} x_{j_1}^\vee\dots \widehat{x_{j_i}^\vee}\dots x_{j_p}^\vee
\end{equation}
$$\text{and}\quad x_{j}\cdot x_{j_1}^\vee\dots x_{j_p}^\vee=0\text{ if }j\notin \{j_1,\dots,j_p\}.$$
Let $f$ be any element of $\pi_1(G)$.
Then written multiplicatively, $f$ is of the form $x_1^{n_1}\dots x_l^{n_l}y$
for some $n_1$,\dots, $n_l\in\mathbb{Z}$
and some element $y$ of $\pi_1(G)$ with torsion.
Then $hur_G:\pi_1(G)\rightarrow \pi_1(G)\otimes\mathbb{Q}\cong H_1(G;\mathbb{Q})$ maps $f$ to $n_1x_1+\dots+n_lx_l+0$
while $hur_G:\pi_{\geq 3}(G)\otimes\mathbb{Q}\rightarrow H_*(G;\mathbb{Q})$
maps $x_j$ to itself for any $l<j\leq r$.

Therefore, by Corollary~\ref{calcul rationel pour les groupes de Lie},
$H_*(\Omega G;\mathbb{Q})\otimes \mathbb{H}_*(G;\mathbb{Q})$
is the tensor product of algebras
$\mathbb{Q}[\pi_1(G)]\otimes \Lambda(s^{-1}\pi_{\geq 3}(G)\otimes\mathbb{Q})
\otimes H^*(G;\mathbb{Q})$ and
\begin{multline*}
\BV_{\Omega G\times G}(fs^{-1}x_{l+1}^{n_{l+1}}\dots s^{-1}x_r^{n_r}\otimes
x_{j_1}^\vee\dots x_{j_p}^\vee)\\
=
 fs^{-1}x_{l+1}^{n_{l+1}}\dots s^{-1}x_r^{n_r}\otimes 
(n_1x_1+\dots+n_lx_l) \cdot x_{j_1}^\vee\dots x_{j_p}^\vee\\
+\sum_{j=l+1}^r n_j
fs^{-1}x_{l+1}^{n_{l+1}}\dots s^{-1}x_j^{n_j-1}\dots s^{-1}x_r^{n_r}\otimes 
x_j \cdot x_{j_1}^\vee\dots x_{j_p}^\vee.
\end{multline*}
Using~(\ref{action de homologie de G sur la cohomologie de G}),
the theorem is proved.
\end{proof}
In~\cite[Theorems 10 and 12]{menichi:stringtopspheres}, we computed
the Batalin-Vilkovisky algebra $\mathbb{H}_*(LS^{2i+1})$ for all odd dimensional spheres. Using this computation, the previous Theorem can be given
the following simple interpretation.
\begin{theorem}
Let $G$ be a path-connected compact Lie group.
Then the Chas-Sullivan Batalin-Vilkovisky algebra on the rational free loop
space homology on $G$, $\mathbb{H}_*(LG;\mathbb{Q})$ is isomorphic to
the tensor product (in the sense of Example~\ref{produit tensoriel BV})
$$
\mathbb{Q}[\pi_1(G)_{\text{tor}}]\otimes \bigotimes_{k=0}^{+\infty}
\mathbb{H}_*(LS^{2k+1};\mathbb{Q})^{\otimes\text{dim }\pi_{2k+1}(G)\otimes\mathbb{Q}}.
$$
Here $\mathbb{Q}[\pi_1(G)_{\text{tor}}]$ denotes the Batalin-Vilkovisky algebra
with trivial operator whose underlying algebra is the group ring on the
torsion subgroup of $\pi_1(G)$.
\end{theorem}
Note that this theorem extends the case of $G=SU(n+1)$ the special unitary
group first proved by Tamanoi in~\cite[Corollary C]{Tamanoi:BVLalg} using a different
method.
\begin{proof}
Recall from~\cite[Theorem 10]{menichi:stringtopspheres}
that
$\mathbb{H}_*(LS^1)\cong\mathbb{Q}[\mathbb{Z}]\otimes \Lambda x^\vee
$ with
$$\BV(x^n\otimes x^\vee)=n(x^n\otimes 1),\quad\BV(x^n\otimes 1)=0$$
for all $n\in\mathbb{Z}$.
Here $x$ denotes a generator of $\pi_1(S^1)\cong \mathbb{Z}$.
Recall from~\cite[Theorem 16]{menichi:stringtopspheres} that
for $i\geq 1$,
$\mathbb{H}_*(LS^{2i+1})\cong\Lambda (s^{-1}x)\otimes\Lambda x^\vee,
$ with
$\BV( s^{-1}x^n\otimes x^\vee)=n(s^{-1}x^{n-1}\otimes 1)$ and
$\BV( s^{-1}x^n\otimes 1)=0$ for all $n\geq 0$.
Here $x$ denote a generator of lower degree $2i+1$.

Let $x_i$'s be the generators defined by Theorem~\ref{homologie rationel des lacets sur les groupes de Lie}.
Let $\Theta$ be the isomorphim of algebras from
$$\mathbb{H}_*(LS^{1};\mathbb{Q})^{\otimes\text{dim }\pi_{1}(G)\otimes\mathbb{Q}}\otimes\mathbb{Q}[\pi_1(G)_{\text{tor}}]\otimes \bigotimes_{k=1}^{+\infty}
\mathbb{H}_*(LS^{2k+1};\mathbb{Q})^{\otimes\text{dim }\pi_{2k+1}(G)\otimes\mathbb{Q}}
$$
to
$$
\mathbb{Q}[\pi_1(G)]\otimes\Lambda (s^{-1}x_j)_{l<j\leq r}\otimes
\Lambda (x_i^\vee)_{1\leq i\leq r}
$$
mapping 

$\bullet$ for $1\leq i\leq l$,
the elements $x_i^{n_i}\otimes 1$ (respectively $x_i^{n_i}\otimes x_i^\vee$) in the $i^{th}$ factor of $\mathbb{H}_*(LS^{1};\mathbb{Q})^{\otimes\text{dim }\pi_{1}(G)\otimes\mathbb{Q}}$ to $x_i^{n_i}\otimes 1\otimes 1$
(respectively $x_i^{n_i}\otimes 1\otimes x_i^\vee$),

$\bullet$ for each $y\in\pi_1(G)_{\text{tor}}$, the element
$y$ to $y\otimes 1\otimes 1$ and

$\bullet$for $l< i\leq r$,
 the elements $s^{-1}x_i^{n_i}\otimes 1$
(respectively $s^{-1}x_i^{n_i}\otimes x_i^\vee$) in the $(i-l)^{th}$ factor of
$\bigotimes_{k=1}^{+\infty}
\mathbb{H}_*(LS^{2k+1};\mathbb{Q})^{\otimes\text{dim }\pi_{2k+1}(G)\otimes\mathbb{Q}}$ to $1\otimes s^{-1}x_i^{n_i}\otimes 1$
(respectively $1\otimes s^{-1}x_i^{n_i}\otimes x_i^\vee$).
 
Explicitly $\Theta$ is the linear isomorphism 
mapping the element
\begin{multline*}
(x_1^{n_1}\otimes 1)\otimes \dots \otimes (x_{j_1}^{n_{j_1}}\otimes x_{j_1}^\vee)
\otimes\dots\otimes (x_l^{n_l}\otimes 1)\otimes y\\
\otimes
(s^{-1}x_{l+1}^{n_{l+1}}\otimes 1)\otimes \dots \otimes (s^{-1}x_{j_p}^{n_{j_p}}\otimes x_{j_p}^\vee)
\otimes\dots\otimes (s^{-1}x_r^{n_r}\otimes 1)
\end{multline*}
to
$
x_1^{n_1}\dots x_l^{n_l}y\otimes s^{-1}x_{l+1}^{n_{l+1}}\dots s^{-1}x_r^{n_r}
\otimes x_{j_1}^\vee\dots x_{j_p}^\vee
$.

In the tensor product of Batalin-Vilkovisky algebras
$$\mathbb{H}_*(LS^{1};\mathbb{Q})^{\otimes\text{dim }\pi_{1}(G)\otimes\mathbb{Q}}\otimes\mathbb{Q}[\pi_1(G)_{\text{tor}}]\otimes \bigotimes_{i=1}^{+\infty}
\mathbb{H}_*(LS^{2i+1};\mathbb{Q})^{\otimes\text{dim }\pi_{2i+1}(G)\otimes\mathbb{Q}},
$$
the operator $\BV$ is given by
\begin{multline*}
\BV(
x_1^{n_1}\otimes 1\otimes \dots \otimes x_{j_1}^{n_{j_1}}\otimes x_{j_1}^\vee
\otimes\dots\otimes x_l^{n_l}\otimes 1\otimes y\\
\otimes
s^{-1}x_{l+1}^{n_{l+1}}\otimes 1\otimes \dots \otimes s^{-1}x_{j_p}^{n_{j_p}}\otimes x_{j_p}^\vee
\otimes\dots\otimes s^{-1}x_r^{n_r}\otimes 1
)
\end{multline*}
\begin{multline*}
=\sum_{i=1, j_i\leq l}^p (-1)^{i-1}
x_1^{n_1}\otimes 1\otimes \dots \otimes x_{j_1}^{n_{j_1}}\otimes x_{j_1}^\vee\otimes
\dots \otimes \BV(x_{j_i}^{n_{j_i}}\otimes x_{j_i}^\vee)
\otimes\dots\otimes x_l^{n_l}\otimes 1\otimes y\\\otimes
s^{-1}x_{l+1}^{n_{l+1}}\otimes 1\otimes \dots \otimes s^{-1}x_{j_p}^{n_{j_p}}\otimes x_{j_p}^\vee
\otimes\dots\otimes s^{-1}x_r^{n_r}\otimes 1\\
+\sum_{i=1,j_i>l}^p (-1)^{i-1}
x_1^{n_1}\otimes 1\otimes \dots \otimes x_{j_1}^{n_{j_1}}\otimes x_{j_1}^\vee
\otimes\dots\otimes x_l^{n_l}\otimes 1\otimes y\\\otimes
s^{-1}x_{l+1}^{n_{l+1}}\otimes 1
\otimes \dots \otimes \BV(s^{-1}x_{j_i}^{n_{j_i}}\otimes x_{j_i}^\vee)
\otimes \dots \otimes s^{-1}x_{j_p}^{n_{j_p}}\otimes x_{j_p}^\vee
\otimes\dots\otimes s^{-1}x_r^{n_r}\otimes 1
\end{multline*}
\begin{multline*}
=\sum_{i=1, j_i\leq l}^p (-1)^{i-1}
x_1^{n_1}\otimes 1\otimes \dots \otimes x_{j_1}^{n_{j_1}}\otimes x_{j_1}^\vee\otimes
\dots \otimes (n_{j_i} x_{j_i}^{n_{j_i}}\otimes 1)
\otimes\dots\otimes x_l^{n_l}\otimes 1\otimes y\\\otimes
s^{-1}x_{l+1}^{n_{l+1}}\otimes 1\otimes \dots \otimes s^{-1}x_{j_p}^{n_{j_p}}\otimes x_{j_p}^\vee
\otimes\dots\otimes s^{-1}x_r^{n_r}\otimes 1\\
+\sum_{i=1,j_i>l}^p (-1)^{i-1}
x_1^{n_1}\otimes 1\otimes \dots \otimes x_{j_1}^{n_{j_1}}\otimes x_{j_1}^\vee
\otimes\dots\otimes x_l^{n_l}\otimes 1\otimes y\\\otimes
s^{-1}x_{l+1}^{n_{l+1}}\otimes 1
\otimes \dots \otimes (n_{j_i} s^{-1}x_{j_i}^{n_{j_i}-1}\otimes 1)
\otimes \dots \otimes s^{-1}x_{j_p}^{n_{j_p}}\otimes x_{j_p}^\vee
\otimes\dots\otimes s^{-1}x_r^{n_r}\otimes 1
\end{multline*}
Therefore
$$
\Theta\circ\BV\circ\Theta^{-1}
(x_1^{n_1}\dots x_l^{n_l}y\otimes s^{-1}x_{l+1}^{n_{l+1}}\dots s^{-1}x_r^{n_r}
\otimes x_{j_1}^\vee\dots x_{j_p}^\vee
)
$$
is given by Theorem~\ref{homologie rationel des lacets sur les groupes de Lie}.
\end{proof}
\section{semi-direct product up to homotopy}\label{semi-direct}
This section on split fibrations is needed in Section~\ref{action du cercle sur les lacets d'un groupe}.
A split short exact sequences of groups gives a semi-direct product.
The following Proposition, which is certainly not new,
is a homotopy version of this fact.
\begin{proposition}
Let $F\buildrel{j}\over\hookrightarrow E\buildrel{p}\over\twoheadrightarrow B$
be a fibration with an homotopy section $\sigma:B\rightarrow E$, $p\circ\sigma\thickapprox id_E$,
such that $j$, $p$ and $\sigma$ are morphims of $H$-groups.
Then $j$ admits a retract $r:E\rightarrow F$ up to homotopy such
that 

$\bullet$ the composite $r\circ\sigma $ is homotopically trivial and

$\bullet$ the composite
$F\times B\buildrel{j\times\sigma}\over\rightarrow E\times E\buildrel{\mu_E}\over\rightarrow E$
and $(r,p):E\rightarrow F\times B$ are homotopy inverse.
\end{proposition}
\begin{proof}
Let $Z$ be any pointed space. Applying the pointed homotopy class functor $[Z,-]$,
the Nomura-Puppe long exact sequence gives the short exact sequence of groups with the splitting
$[Z,\sigma]$,
$$
1\rightarrow [Z,F]\buildrel{[Z,j]}\over\hookrightarrow[Z,E]\buildrel{[Z,p]}\over\twoheadrightarrow[Z,B]\rightarrow 1.
$$
Therefore the group $[Z,E]$ is isomorphic to the semi-direct product of groups $[Z,F]\rtimes [Z,B]$.
More precisely:

The map $E\rightarrow E$ sending $e\mapsto e\cdot\sigma\circ p(e)^{-1}$
induces the map
$$
[Z,\mu_E\circ(id_E,inv\circ \sigma\circ p)]:[Z,E]\rightarrow [Z,E].
$$
Let $R:[Z,E]\rightarrow [Z,F]$ be the unique map such that the composite $[Z,j]\circ R$ coincides with
$[Z,\mu_E\circ(id_E,inv\circ \sigma\circ p)]$. The map $R$ is a retract of $[Z,j]$.

The bijection $[Z,F]\times [Z,B]=[Z,F\times B]\rightarrow [Z,E]$ is given by
$\mu_{[Z,E]}\circ ([Z,j]\times [Z,\sigma])=[Z,\mu_E\circ (j\times\sigma)]$.
Its inverse $[Z,E]\rightarrow [Z,F]\times [Z,B]$ is the map $(R,[Z,p])$.
By Yoneda lemma, the natural transformation $R:[Z,E]\rightarrow [Z,F]$ can be uniquely written as a $[Z,r]$
where $r$ is a pointed map from $E$ to $F$.
By Yoneda Lemma again, the composite
$F\times B\buildrel{j\times\sigma}\over\rightarrow E\times E\buildrel{\mu_E}\over\rightarrow E$
and $(r,p):E\rightarrow F\times B$ are homotopy inverse, $r\circ j\thickapprox Id_F$ and
$r\circ\sigma $ is homotopically trivial.
\end{proof}
\bibliography{Bibliographie}
\bibliographystyle{amsplain}
\end{document}